\documentclass[11pt]{article}

\usepackage{amssymb} 
\usepackage{amsthm} 
\usepackage{latexsym}
\usepackage{amsmath} 
\usepackage[utf8]{inputenc} 

\def\n{\nabla}

\def\F{{\mathcal F}}
\def\P{{\mathcal P}}
\def\I{{\mathcal I}}
\def\R{{\mathcal R}}
\def\S{{\mathcal S}}

\theoremstyle{plain} 
\newtheorem{theorem}{Theorem}[section]
\newtheorem{proposition}[theorem]{Proposition}

\newtheorem{corollary}[theorem]{Corollary}
\theoremstyle{definition}
\newtheorem{definition}[theorem]{Definition}
\newtheorem{lemma}[theorem]{Lemma}
\newtheorem{remark}[theorem]{Remark}
\theoremstyle{remark}

\newcommand{\bp}{\begin{proof}\;}
\newcommand{\ep}{\end{proof}}

\title{Projectively flat Finsler manifolds with infinite dimensional
  holonomy}

\author{Zolt\'an Muzsnay and P\'eter T. Nagy}

\date{Institute of Mathematics, University of Debrecen\\
  H-4010 Debrecen, Hungary, P.O.B. 12\\
  \bigskip {\small{\it E-mail}: {\tt {}muzsnay@science.unideb.hu}, {\tt
    {}petert.nagy@science.unideb.hu}}}
\begin{document}

\maketitle

\begin{abstract}
  Recently, we developed a method for the study of holonomy properties
  of non-Riemannian Finsler manifolds and obtained that the holonomy
  group can not be a compact Lie group, if the Finsler manifold of
  dimension $> 2$ has non-zero constant flag curvature. The purpose of
  this paper is to move further, exploring the holonomy properties of
  projectively flat Finsler manifolds of non-zero constant flag
  curvature.  We prove in particular that projectively flat Randers and Bryant-Shen
  manifolds of non-zero constant flag curvature have infinite
  dimensional holonomy group.
\end{abstract}



\footnotetext{This research was supported by the Hungarian Scientific
  Research Fund (OTKA) Grant K 67617.}

\section{Introduction}

In recent papers \cite{Mu_Na}, \cite{MuNa} we have developed a method
for the study of holonomy properties of non-Riemannian Finsler
manifolds. Particularly, we obtained in \cite{Mu_Na}, that the holonomy
group cannot be a compact Lie group if the Finsler manifold of dimension
$> 2$ has non-zero constant flag curvature. We described the first
example of a Finsler manifold with infinite dimensional holonomy group,
namely a left invariant Berwald-Mo\'or metric on the 3-dimensional
Heisenberg group. 

The purpose of the present paper is to investigate families of Finsler
manifolds with interesting geometric structure wich have infinite
dimensional holonomy group. From the viewpoint of non-Euclidean geometry
the most important Riemann-Finsler manifolds are the projectively flat
spaces of constant flag curvature. We will turn our attention to
non-Riemannian projectively flat Finsler manifolds of non-zero constant
flag curvature. \par
There are many examples of non-Riemannian projectively flat Finsler
manifolds of non-zero constant flag curvature. Their classification is
related to the smooth version of Hilbert's Fourth Problem: characterize
(not necessarily reversible) distance functions on an open subset in
$\mathbb R^n$ such that straight lines are the shortest paths.  The
famous Beltrami's theorem states that the locally projectively flat
Riemannian manifolds are exactly the manifolds of constant
curvature. But for Finsler manifolds Beltrami's theorem is not true. In
fact, any projectively flat Finsler manifold has scalar flag curvature,
but there are Finsler manifolds of constant flag curvature, which are
not projectively flat. The constancy of the flag curvature is a very
restrictive property for complete projectively flat Finsler manifolds,
cf.  \cite{Fk1}, \cite{Fk2}, \cite{Fk3}, \cite{Br1}, \cite{Br2},
\cite{Br3}, but there are many non-complete examples defined on open
domains in $\mathbb R^n$, constructed and studied by Z. Shen,
cf. \cite{ChSh}, \cite{Shen1}, \cite{Shen2}, \cite{Shen3},
\cite{Shen4}.  \par
We will consider the following classes of locally projectively flat
non-Riemannian Finsler manifolds of non-zero constant flag curvature:
\vspace{-3pt}
\begin{enumerate}
\item \emph{Randers manifolds,} (cf. \cite{Shen3}), \vspace{-5pt}
\item \emph{manifolds having a $2$-dimensional subspace in the tangent
    space at some point, on which the Finsler norm is an Euclidean
    norm,} (cf. \cite{Shen4}, Theorem 1.2, p. 1715,) \vspace{-5pt}
\item \emph{manifolds having a $2$-dimensional subspace in the tangent
    space at some point, on which the Finsler norm and the projective
    factor are linearly dependent,} (cf. \cite{Shen4}, Example 7.1,
  pp. 1725 - 1727). %\vspace{-5pt}
\end{enumerate}
The first class consists of positively complete Finsler manifolds of
negative curvature, the second class contains a large family of (not
necessarily complete) Finsler manifolds of negative curvature, and the
third class contains a large family of not necessarily complete Finsler
manifolds of positive curvature. The metrics belonging to these classes
can be considered as (local) generalizations of a one-parameter family
of complete Finsler manifolds of positive curvature defined on $S^2$ by
R. Bryant in [Br1], [Br2] and on $S^n$ by Z. Shen in \cite{Shen4},
Example 7.1.  We prove that the holonomy group of Finsler manifolds
belonging to these classes and satisfying some additional technical
assumption is infinite dimensional.

\section{Preliminaries}

Throughout this article, $M$ denotes a $C^\infty$ manifold, ${\mathfrak
  X}^{\infty}(M)$ denotes the vector space of smooth vector fields on
$M$ and ${\mathsf {Diff}}^\infty(M)$ denotes the group of all
$C^\infty$-diffeomorphism of $M$ with the $C^\infty$-topology.

\subsubsection*{Spray manifold, horizontal distribution, canonical
  connection, parallelism}

A \emph{spray} on a manifold $M$ is a smooth vector field $\S$ on
$\hat T M := TM\setminus\! \{0\}$ expressed in a standard coordinate
system $(x^i,y^i)$ on $TM$ as
\begin{equation} 
  \label{spray}
  \S = y^i\frac{\partial}{\partial x^i} - 2 \Gamma^i(x,y) 
  \frac{\partial}{\partial y^i},
\end{equation}
where $\Gamma^i(x,y)$ are local functions on $TM$ satisfying 
\[\Gamma^i(x,\lambda y) = \lambda^2\Gamma^i(x,y),\quad \lambda>0.\] 
A manifold $M$ with a spray $\S$ is called a \emph{spray manifold} $(M,\S)$.\\
A curve $c(t)$ is called \emph{geodesic} if its coordinate functions
$c^i(t)$ satisfy the differential equations
\begin{equation} \label{geod}
\ddot{c}^i(t) + \Gamma^i(c(t),\dot{c}(t)) = 0,
\end{equation}
where the functions $\Gamma^i(x,y)$ 
are called the \emph{geodesic coefficients} the spray manifold $(M,\S)$.
The associated \emph{homogeneous (nonlinear) parallel translation}
$\tau_{c}:T_{c(0)}M\to T_{c(1)}M$ along a curve $c(t)$ is defined by
parallel vector fields $X(t)=X^i(t)\frac{\partial}{\partial x^i}$ along
$c(t)$ satisfying
\begin{equation} \label{eq:D} D_{\dot c} X (t):=\Big(\frac{d X^i(t)}{d
    t}+ \Gamma^i_j(c(t),X(t))\dot c^j(t)\Big) \frac{\partial}{\partial
    x^i} =0,\quad \text{where}\quad
  \Gamma^i_j=\frac{\partial\Gamma^i}{\partial y^j}.
\end{equation}
Let $(TM,\pi ,M)$ and $(TTM,\tau ,TM)$ denote the first and the second
tangent bundle of the manifold $M$, respectively. The \emph{horizontal
  distribution} ${\mathcal H}TM \!  \subset\! TTM$ associated to the
spray manifold $(M, \S)$ is the image of the horizontal lift which
is a vector space isomorphism $l_y\colon T_xM\to{\mathcal H}_yTM$ for
every $x\in M$ and $y\in T_xM$ defined by
\begin{equation} 
  \label{eq:lift} 
  l_y  \Big(  \frac{\partial}{\partial
    x^i}  \Big)=  \frac{\partial}{\partial x^i}
  -\Gamma_i^k(x,y)\frac{\partial}{\partial y^k}.
\end{equation}
Then a vector field $X(t)$ along a curve $c(t)$ is parallel if and only
if it is a solution of the differential equation
\begin{math}
  \frac{d}{dt}X(t)=l_{X(t)}(\dot c(t)).
\end{math}
\\[1ex]
If ${\mathcal V}TM \!  \subset\!  TTM$ is the vertical distribution on
$TM$ defined by ${\mathcal V}_yTM:=\mathrm{Ker} \, \pi_{*,y}$, then we
have the decomposition $T_yTM = {\mathcal H}_yTM \oplus {\mathcal
  V}_yTM$.
\\[1ex]
Let $(\hat{\mathcal V}TM,\tau,\hat TM)$ be the vertical bundle over
$\hat TM:= TM\setminus\!\{0\}$. We denote by $\hat{\mathfrak
  X}^{\infty}(TM)$ the vector space of smooth sections of the bundle
$(\hat{\mathcal V}TM,\tau,\hat T M)$.  The \emph{horizontal covariant
  derivative} of a section $\xi\in\hat{\mathfrak X}^{\infty}(TM)$ by a
vector field $X\in{\mathfrak X}^{\infty}(M)$ is given by $\nabla_X\xi :=
[h(X),\xi].$ We can express the horizontal covariant derivative of the
section $\xi(x,y)=\xi^i(x,y)\frac{\partial}{\partial y^i}$ by the vector
field $X(x) = X^i(x)\frac {\partial}{\partial x^i}$ as
\begin{equation}\label{covder}
  \nabla_X\xi = \left(\frac {\partial\xi^i(x,y)}{\partial x^j} - 
    \Gamma_j^k(x,y)\frac{\partial \xi^i(x,y)}{\partial y^k} + 
    \Gamma^i_{j k}(x,y)\xi^k(x,y)\right)X^j\frac {\partial}{\partial y^i},\end{equation}
where \[\Gamma^i_{j k}(x,y) := \frac{\partial \Gamma_j^i(x,y)}{\partial y^k}.\]
Let  $(\pi^{*}TM,\bar{\pi},\hat T M)$ be the pull-back bundle of $(\hat T M,\pi ,M)$ by the map $\pi:TM\to M$. 
Clearly, the mapping
\begin{equation} 
  \label{iden}
  (x,y,\xi^i\frac {\partial}{\partial
    y^i})\mapsto(x,y,\xi^i\frac {\partial}{\partial x^i})
  :\;\hat{\mathcal V}TM\rightarrow \pi^{*}TM
\end{equation} 
is a canonical bundle isomorphism.  In the following we will use the
isomorphism (\ref{iden}) for the identification of these bundles.\\[1ex]
The \emph{curvature tensor} field $K_{(x,y)} =
K^i_{jk}(x,y)dx^j\otimes dx^k\otimes \frac{\partial}{\partial x^i}$ on
the pull-back bundle $(\pi^{*}TM,\bar{\pi},\hat T M)$ of the spray manifold $(M,\S)$ 
in a local coordinate system is given by 
\begin{equation} \label{spraycurv}
  K^i_{jk}(x,y) =  \frac{\partial \Gamma^i_j(x,y)}{\partial x^k} -
  \frac{\partial \Gamma^i_k(x,y)}{\partial x^j} + \Gamma_j^m(x,y)
  \Gamma^i_{k m}(x,y) - \Gamma_k^m(x,y)
  \Gamma^i_{j m}(x,y). 
\end{equation}

\subsubsection*{Finsler manifold, canonical connection, Berwald
  connection}

A \emph{Finsler manifold} is a pair $(M,\mathcal F)$, where $M$ is an
$n$-dimensional smooth manifold and $\F\colon TM \to \mathbb{R}$ is a
continuous function, smooth on $\hat T M := TM\setminus\! \{0\}$, such
that its restriction ${\mathcal F}_x={\mathcal F}|_{_{T_xM}}$ for any
$x\in M$ is a positively 1-homogeneous function and the symmetric
bilinear form
\begin{displaymath}
  g_{x,y} \colon (u,v)\ \mapsto \ g_{ij}(x, y)u^iv^j=\frac{1}{2}
  \frac{\partial^2 \mathcal F^2_x(y+su+tv)}{\partial s\,\partial
    t}\Big|_{t=s=0}
\end{displaymath}
is positive definite at every $y\in \hat T_xM$.
\\[1ex]
We call $\F$ the Finsler function, $g_{x,y}$ the metric tensor of the
Finsler manifold $(M,\mathcal F)$.  The Finsler function is called
\emph{absolutely homogeneous} at $x\in M$, if $\F_x(\lambda
y)\!=\!|\lambda|\F_x(y)$. If $\F$ is absolutely homogeneous at every
$x\in M$, then the Finsler manifold $(M,\mathcal F)$ is
\emph{reversible}.
\\[1ex]
The associated spray is locally given by
$\S=y^i\frac{\partial}{\partial x^i} -
2G^i(x,y)\frac{\partial}{\partial y^i}$, where the functions
\begin{equation}\label{eq:G_i}
  G^i(x,y):= \frac{1}{4}g^{il}(x,y)\Big(2\frac{\partial g_{jl}}{\partial x^k}(x,y) -
  \frac{\partial g_{jk}}{\partial x^l}(x,y) \Big) y^jy^k
\end{equation}
are the \emph{geodesic coefficients} the Finsler manifold
$(M,\mathcal F)$.  The corresponding homogeneous (nonlinear) parallel
translation $\tau_{c}:T_{c(0)}M\to T_{c(1)}M$ along a curve $c(t)$ is
called the \emph{canonical homogeneous (nonlinear) parallel translation}
of the Finsler manifold $(M,\mathcal F)$. The horizontal covariant
derivative with respect to the spray associated to the Finsler manifold
$(M,\mathcal F)$ is called the \emph{horizontal Berwald covariant
  derivative}. If we define
\[\nabla_X  \phi= \left(\frac {\partial\phi}{\partial x^j} - 
  G_j^k(x,y)\frac{\partial \phi(x,y)}{\partial y^k}\right)X^j\] for a
smooth function $\phi:\hat TM\to \mathbb R$, the horizontal Berwald
covariant derivation (\ref{covder}) can be extended to the tensor bundle
over $(\pi^{*}TM,\bar{\pi},\hat T M)$.
\\[1ex]
The \emph{Riemannian curvature tensor} field $R_{(x,y)} =
R^i_{jk}(x,y)dx^j\otimes dx^k\otimes \frac{\partial}{\partial x^i}$ on
the pull-back bundle $(\pi^{*}TM,\bar{\pi},\hat T M)$ is
\begin{displaymath}
  R^i_{jk}(x,y) =  \frac{\partial G^i_j(x,y)}{\partial x^k} -
  \frac{\partial G^i_k(x,y)}{\partial x^j} + G_j^m(x,y)
  G^i_{k m}(x,y) - G_k^m(x,y)
  G^i_{j m}(x,y). 
\end{displaymath} 
The manifold is called of constant flag curvature $\lambda\in{\mathbb
  R}$, if for any $x\in M$ the local expression of the Riemannian
curvature is
\begin{equation}
  \label{gorb}
  R^i_{jk}(x,y) = \lambda\left(\delta_k^ig_{jm}(x,y)y^m - \delta_j^ig_{km}(x,y)y^m
  \right).
\end{equation}
In this case the flag curvature of the Finsler manifold
(cf.~\cite{ChSh}, Section 2.1 pp. 43-46) does not depend on the point,
nor on the 2-flag.
\\
A Finsler function $\F$ on an open subset $D \subset \mathbb R^n$ is
said to be \emph{projectively flat}, if all geodesic are straight lines in
$D$. A Finsler manifold is said to be \emph{locally projectively flat},
if at any point there is a local coordinate system $(x^i)$ in which $\F$
is projectively flat.
\\
Let $(x^1,\dots ,x^n)$ be a local coordinate system on $M$ corresponding
to the canonical coordinates of the Euclidean space which is
projectively related to $(M, \F)$. Then the geodesic coefficients
(\ref{eq:G_i}) are of the form
\begin{equation}
\label{eq:proj_flat_G_i}
  G^i(x,y) = \P(x,y)y^i, \quad  G^i_k = \frac{\partial
    \P}{\partial y^k}y^i + \P\delta^i_k,\quad G^i_{kl} = \frac{\partial^2
    \P}{\partial y^k\partial y^l}y^i + \frac{\partial \P}{\partial
    y^k}\delta^i_l + \frac{\partial \P}{\partial y^l}\delta^i_k. 
\end{equation}
where $\P$ is a 1-homogeneous function in $y$, called the projective factor 
of $(M,\F)$. Clearly, any $2$-plane in this coordinate system 
$(x^1,\dots ,x^n)$ is a totally geodesic submanifold of $(M,\F)$. 
\begin{remark}\label{linear} The canonical homogeneous parallel translation 
$\tau_{c}:T_{c(0)}M\to T_{c(1)}M$ in a locally projectively flat Finsler 
manifold $(M,\mathcal F)$ along curves $c(t)$ contained in the domain of the 
coordinate system $(x^1,\dots ,x^n)$ are linear maps if and only if the 
projective factor $\P(x,y)$ is a linear function in $y$. Hence the 
non-linearity in $y$ of the projective factor implies that the locally 
projectively flat Finsler manifold is non-Riemannian. 
\end{remark}

Projectively flat Randers manifolds with constant flag curvature were 
classified by Z. Shen in \cite{Shen3}. He proved that any projectively 
flat Randers manifold $(M, \F)$ with non-zero constant flag curvature has 
negative curvature. These metrics can be normalized by a constant factor  
so that the curvature is $-\frac14$. In this case $(M, \F)$  is isometric 
to the Finsler manifold defined by the metric function 
\begin{equation}
  \label{projective1} \F(x,y) = \frac{\sqrt{|y|^2 - \left(|x|^2|y|^2 - 
  \langle x,y\rangle^2\right)}}{1 - |x|^2}\pm\left(\frac{\langle x,y\rangle}
  {1 - |x|^2} + \frac{\langle a,y\rangle}{1 + \langle a,x\rangle}\right) 
\end{equation}
on the unit ball $\mathbb D^n \subset \mathbb R^n$, where $a\in \mathbb
R^n$ is any constant vector with $|a|<1$. According to Lemma 8.2.1 in
\cite{ChSh}, p.155, the projective factor $\P(x,y)$ can be computed by
the formula
\begin{displaymath}
  \P(x,y) = \frac{1}{2\F}\frac{\partial\F}{\partial x^i}y^i.
\end{displaymath}
An easy calculation yields
\begin{displaymath}
  \pm \frac{\partial\F}{\partial x^i}y^i = 
  \left(\! \frac{\sqrt{|y|^2 - \left(|x|^2|y|^2 - \langle
          x,y\rangle^2\right)} \pm \langle x,y\rangle}{1 -
      |x|^2}\right)^{\!\!2} - \left(\frac{\langle a,y\rangle}{1 + \langle
      a,x\rangle}\right)^{\!\!2},
\end{displaymath}
hence
\begin{equation}
  \label{projective2}
  \P(x,y) =  \frac12 \left(\!\!\frac{\pm\sqrt{|y|^2 - \left(|x|^2|y|^2 
          - \langle x,y\rangle^2\right)} + \langle x,y\rangle}{1 - |x|^2} 
    - \frac{\langle a,y\rangle}{1 + \langle a,x\rangle}\right).
\end{equation}

\subsubsection*{Holonomy group}

Let $(M,\mathcal F)$ be an $n$-dimensional Finsler manifold.  We denote
by $(\I M,\pi,M)$ the \emph {indicatrix bundle} of $(M,\mathcal F)$, the
\emph{indicatrix} $\I_xM$ at $x \in M$ is the compact hypersurface
\begin{displaymath}
  \I_xM:= \{y \in T_xM ; \ \mathcal F(y) = 1\},
\end{displaymath}
of $T_xM$ diffeomorphic to the standard $(n-1)$-sphere.
\\[1ex]
The homogeneous (nonlinear) parallel translation $\tau_{c}:T_{c(0)}M\to
T_{c(1)}M$ along a curve $c:[0,1]\to \R$ preserves the value of the
Finsler function, hence it induces a map $\tau_{c}\colon \I_{c(0)}M
\longrightarrow \I_{c(1)}M$
between the indicatrices. \\[1ex]
The group of diffeomorphisms ${\mathsf{Diff}^{\infty}}({\I}_xM)$ of the
indicatrix ${\I}_xM$ is a regular infinite dimensional Lie group modeled
on the vector space ${\mathfrak X}^{\infty}({\I}_xM)$. In this category
of groups one can define the exponential mapping and the group structure
is locally determined by the Lie algebra. The Lie algebra of
${\mathsf{Diff}^{\infty}}({\I}_xM)$ is ${\mathfrak X}^{\infty}({\I}_xM)$
equipped with the negative of the usual Lie bracket.
\\[1ex]
The \emph{holonomy group} $\mathsf{Hol}_x(M)$ of a Finsler space
$(M,\F)$ at a point $x\in M$ is the subgroup of the group of
diffeomorphisms ${\mathsf{Diff}^{\infty}}({\I}_xM)$ of the indicatrix
${\I}_xM$ generated by (nonlinear) parallel translations of ${\I}_xM$
along piece-wise differentiable closed curves initiated at the point
$x\in M$.
\\[1ex]
Clearly, the holonomy groups at different points of $M$ are isomorphic.

\subsubsection*{Infinitesimal holonomy algebra} 

A vector field $\xi\in {\mathfrak X}^{\infty}({\I}M)$ on the indicatrix
bundle $\I M$ is a \emph{curvature vector field} of the Finsler manifold
$(M, \F)$, if there exist vector fields $X, Y\in {\mathfrak
  X}^{\infty}(M)$ on $M$ such that $\xi = R(X,Y)$.
\\[1ex]
If $x\in M$ is fixed and $X, Y\in T_xM$, then the vector field $y\to
R(X,Y)(x,y)$ on ${\I}_x M$ is a \emph{curvature vector field at $x$}
(see \cite{Mu_Na}).\\[1ex]
The Lie algebra $\mathfrak{R}(M)$ of vector fields generated by the
curvature vector fields of $(M, \F)$ is called the \emph{curvature
  algebra} of the Finsler manifold $(M, \F)$. For a fixed $x\in M$ the
Lie algebra $\mathfrak{R}_x(M)$ of vector fields generated by the
curvature vector fields at $x$ is called the \emph{curvature algebra
  at  $x$}.\\[1ex]
The \emph{infinitesimal holonomy algebra} of the Finsler manifold $(M,
\mathcal F)$ is the smallest Lie algebra $\mathfrak{hol}^{*}(M)$ of
vector fields on the indicatrix bundle $\I M$ satisfying the properties
\begin{enumerate}
\item[(i)] any curvature vector field $\xi$ belongs to
  $\mathfrak{hol}^{*}(M)$, \vspace{-5pt}
\item[(ii)] if $\xi, \eta\in \mathfrak{hol}^{*}(M)$ then
  $[\xi,\eta]\in\mathfrak{hol}^{*}(M)$, \vspace{-5pt}
\item[(iii)] if $\xi\in\mathfrak{hol}^{*}(M)$ and $X\in {\mathfrak
    X}^{\infty}(M)$ then the horizontal Berwald covariant derivative
  $\nabla_{\!\! X}\xi$ also belongs to
  $\mathfrak{hol}^{*}(M)$. \vspace{-5pt}
\end{enumerate} 
The \emph{infinitesimal holonomy algebra at a point $x\in M$} is the the
Lie algebra
\begin{displaymath}
  \mathfrak{hol}^{*}_x(M) \! := \! 
  \big\{\, \xi \big|_{\I_xM} \ ; \ \xi\in\mathfrak{hol}^{*}(M) \, \big\}
  \subset {\mathfrak  X}^{\infty}({\I}_xM)
\end{displaymath}
of vector fields on the indicatrix $\I_x M$.
\\[1ex]
One has $\mathfrak{R}(M)\subset\mathfrak{hol}^{*}(M)$ and
$\mathfrak{R}_x(M)\subset\mathfrak{hol}^{*}_x(M)$ for any $x\in M$ (see
\cite{MuNa}).
\\[1ex]
Let $H$ be a subgroup of the diffeomorphism group
$\mathsf{Diff}^{\infty}(M)$ of a differentiable manifold $M$.
%\\[1ex]
A vector field $X\!\in\!{\mathfrak X}^{\infty}(M)$ is called
\emph{tangent to} $H\subset\mathsf{Diff}^{\infty}(M)$ if there exists a
${\mathcal C}^1$-differentiable $1$-parameter family $\{\Phi(t)\in
H\}_{t\in\mathbb R}$ of diffeomorphisms of $M$ such that
$\Phi(0)=\mathsf{Id}$ and \( \frac{\partial\Phi(t)}{\partial
  t}\big|_{t=0}=X.\) A Lie subalgebra $\mathfrak g$ of ${\mathfrak
  X}^{\infty}(M)$ is called \emph{tangent
  to} $H$, if all elements of $\mathfrak g$ are tangent vector fields to $H$. \\[1ex]
A subgroup $H$ of the diffeomorphism group $\mathsf{Diff}^{\infty}(M)$
of a manifold $M$ will be called \emph{infinite dimensional}, if $H$ has
an infinite dimensional tangent
Lie algebra of vector fields. \\[1ex]
The following assertion will be an important tool in the next
discussion:
\\[1ex]
\emph{The infinitesimal holonomy algebra $\mathfrak{hol}^{*}(x)$ at any point 
$x\in M$ is tangent to the holonomy group $\mathsf{Hol}(x)$.}
(Theorem 6.3 in \cite{MuNa} 
Hence we have: 
\\[1ex]
\emph{If the infinitesimal holonomy algebra $\mathfrak{hol}^{*}(x)$ at a point 
$x\in M$ is infinite dimensional then the holonomy group $\mathsf{Hol}(x)$ is 
infinite dimensional.}

\section{Holonomy of projective Finsler surfaces of constant curvature}

A Finsler manifold $(M, \F)$ of dimension $2$ is called \emph{Finsler
surface}. In this case the indicatrix is 1-dimensional at any point $x\in M$, 
hence the curvature vector fields at $x\in M$ are proportional to any given 
non-vanishing curvature vector field. It follows that the curvature algebra 
$\mathfrak{R}_x(M)$ has a simple structure: it is at most 1-dimensional and 
commutative. Even in this case, the infinitesimal holonomy algebra
$\mathfrak{hol}^{*}_x(M)$ can be higher dimensional, or potentially
infinite dimensional. For the investigation of such examples we use a
classical result of S.~Lie on the classification of Lie group actions on
one-manifolds (cf.~\cite{Brou} or \cite{Olver}, pp.~58-62):
\\[1ex]
{\it If a finite-dimensional connected Lie group acts on a
  $1$-dimensional manifold without fixed points, than its dimension is
  less than $4$.}
\begin{proposition}
  \label{sec:prop-1}
  If the infinitesimal holonomy algebra $\mathfrak{hol}^{*}_x(M)$
  contains $4$ simultaneously non-vanishing $\mathbb R$-linearly
  independent vector fields, then the holonomy group $\mathsf{Hol}_x(M)$
  is an infinite dimensional subgroup of
  ${\mathsf{Diff}^{\infty}}({\I}_xM)$.
\end{proposition}
\begin{proof} 
  Indeed, in this case the holonomy group acts on the 1-dimensional
  indicatrix without fixed points.  If it would be finite-dimensional
  then its dimension and the dimension of its Lie algebra should
  be less than $4$, hence the assertion is proved.
\end{proof}
Let $(M, \F)$ be a locally projectively flat Finsler surface of non-zero 
constant curvature, let $(x^1,x^2)$ be a local coordinate system
centered at $x\in M$, corresponding to the canonical coordinates of the
Euclidean space which is projectively related to $(M, \F)$ and let
$(y^1,y^2)$ be the induced coordinate system in the tangent plane
$T_xM$. 

In the sequel we identify the tangent plane $T_xM$ with $\mathbb R^2$
with help of the coordinate system $(y^1, y^2)$. We will use the
euclidean norm $||(y^1, y^2)|| = \sqrt{(y^1)^2+(y^2)^2}$ of $\mathbb
R^2$ and the corresponding polar coordinate system $(e^r,t)$, too.

Let $\varphi(y^1,y^2)$ be a positively 1-homogeneous function on
$\mathbb R^2$ and let $r(t)$ be the $2\pi$-periodic smooth function
$r:\mathbb R \to\mathbb R$ determined by
\begin{equation}
  \label{varphi}
  \varphi(e^{r(t)}\!\cos t,e^{r(t)}\!\sin t)=1\quad \text{or} \quad
  \varphi(y^1,y^2) = e^{-r(t)}\sqrt{(y^1)^2\!+\!(y^2)^2},  \end{equation}
where\;
\begin{math}
  \cos t \!= \!\frac{y^1}{\sqrt{(y^1)^2+(y^2)^2}},
\end{math}
\;
\begin{math}
  \sin t \!=\!  \frac{y^2}{\sqrt{(y^1)^2+(y^2)^2}},
\end{math}
and $\tan t \!=\! \frac{y^2}{y^1}$, i.e. the level set $\{\varphi(y^1\!,y^2)\equiv1\}$ of the
1-homogeneous function $\varphi$ in $\mathbb R^2$ is given by the the parametrized
curve $t\to (e^{r(t)}\cos t,e^{r(t)}\sin t).$  

Since the curvature $\kappa$ of a smooth curve \ $t\to (e^{r(t)}\cos t,e^{r(t)}\sin t)$ in
  $\mathbb R^2$ is   
  \begin{equation}\label{kappa}
    \kappa = - \frac{e^r}{\sqrt{{\dot r}^2 + 1}}(\ddot r - \dot r^2 - 1 ),
  \end{equation} 
  the vanishing of the expression $\ddot
r - \dot r^2 - 1$ means the infinitesimal linearity of the corresponding
positively homogeneous function in $\mathbb R^2$.

\begin{definition}
  Let $\varphi(y^1,y^2)$ be a positively 1-homogeneous function on
  $\mathbb R^2$ and let $\kappa(t)$ be the curvature of the curve $t\to
  (e^{r(t)}\cos t,e^{r(t)}\sin t)$ defined by the equations
  (\ref{varphi}). We say that $\varphi(y^1,y^2)$ is \emph{strongly
    convex}, if $\kappa(t) \neq 0$ for all $t\in \mathbb R$.
\end{definition}
Conditions (A), (B), (C) in the following theorem imply that the projective factor $\P$ 
at ${x_0}\in M$ is a non-linear function, and hence, according to Remark \ref{linear}, 
$(M, \F)$ is a non-Riemannian Finsler manifold.

\begin{theorem} 
  \label{prop:proj_flat_const_curv} 
  Let $(M, \F)$ be a projectively flat Finsler surface of non-zero constant curvature. 
  Assume that there exists a point $x_0\in M$ such that one of the following conditions hold 
  \begin{enumerate}
  \item[\emph{(A)}] $\F$ induces a scalar product on $T_{x_0}M$ and the
    projective factor $\P$ at $x_0$ is a strongly convex positively
    1-homogeneous function,
  \item[\emph{(B)}] $\F(x_0,y)$ is a strongly convex absolutely
    1-homogeneous function on $T_{x_0}M$, and the projective factor
    $\P({x_0},y)$ on $T_{x_0}M$ satisfies
    $\P({x_0},y)=c\!\cdot\!\F(x_0,y)$ with $0\neq c\in\mathbb R$,
\item[\emph{(C)}] there is a projectively related Euclidean coordinate
  system of $(M, \F)$ centered at ${x_0}$ and one has
  \begin{equation}
    \label{rand} 
    \F(0,y) = |y| \pm \langle a,y\rangle\quad \text{and}\quad \P(0,y) =  
    \frac12\left(\pm |y| - \langle a,y\rangle\right).\end{equation}
  \end{enumerate}
  Then the holonomy group $\mathsf{Hol}_{x_0}(M)$ is 
  infinite  dimensional.
\end{theorem} 
\begin{proof}
  Let $(M, \F)$ be a Finsler surface of constant flag curvature covered
  by a coordinate system $(x^1,x^2)$. Assume that the vector fields
  $U\!=\!U^i\frac{\partial}{\partial x^i}$,
  $V\!=\!V^i\frac{\partial}{\partial x^i}\in {\mathfrak X}^{\infty}(M)$
  have constant coordinate functions.  Let $\xi = R(U,V)$ be the
  corresponding curvature vector field. \\
  Since $(M, \F)$ is of constant flag curvature, we can write
  \[R^i_{jk}(x,y) = \lambda\left(\delta_j^ig_{km}(x,y)y^m -
    \delta_k^ig_{jm}(x,y)y^m \right), \quad \text{with}\quad \lambda =
  const.\] It is well known that the horizontal Berwald covariant
  derivative $\nabla_WR$ of the tensor field $R =
  R^i_{jk}(x,y)dx^j\wedge dx^k\frac{\partial}{\partial x^i}$
  vanishes. Indeed, Lemma 6.2.2, p. 85 in \cite {Shen1} yields
  $\nabla_wg_{(x,y)}(u,v) = -2 L(u,v,w)$ for any $u,v,w\in
  T_xM$. Moreover $\nabla_W y = 0$,\; $\nabla_W \text{Id}_{TM} = 0$ for
  any vector field $W\in{\mathfrak X}^{\infty}(M)$, and
  $L_{(x,y)}(y,v,w)= 0$ (cf. equation 6.28, p. 85 in
  \cite{Shen1}). Hence we obtain $\nabla_WR = 0$.
  \\
  Since the curvature tensor field is skew-symmetric, $R_{(x,y)}$ acts
  on the one-dimensional wedge product $T_xM\wedge T_xM$. The covariant
  derivative $\n_W\xi$ of the curvature vector field
  \begin{math}
    \xi \!= \!R(U,V) \!=\! \frac12 R(U\otimes V \!-\! V\otimes U) \!=\!
    R(U\!\wedge\! V)
  \end{math}
  can be written in the form
  \begin{displaymath}
    \n_W\xi = R\left(\n_W(U\wedge V)\right) = R(\n_WU\wedge V+U\wedge
    \n_WV).
  \end{displaymath}
  We have
  \begin{math}
    U\wedge V = \frac12 \left(U^1V^2-U^2V^1\right)
    \frac{\partial}{\partial x^1}\wedge\frac{\partial}{\partial x^2}
  \end{math}
  and hence
  \begin{equation}
    \label{dercurv} 
    \n_W\xi = (U^1V^2-V^1U^2) W^k R\Big(\n_k\left(\tfrac{\partial}{\partial x^1} 
      \wedge\tfrac{\partial}{\partial x^2}\right)\Big),
  \end{equation}
  where $\n_k\xi:=
  \n_{\!\!\!\frac{\partial}{\partial x^k}}\xi$.  Since
  \begin{alignat*}{1}
    \n_k\left(\tfrac{\partial}{\partial x^1}
      \wedge\tfrac{\partial}{\partial x^2}\right)& =
    \left(\n_k\tfrac{\partial}{\partial x^1}\right)
    \wedge\tfrac{\partial}{\partial x^2} + \tfrac{\partial}{\partial
      x^1} \wedge\left(\n_k\tfrac{\partial}{\partial x^2}\right) =
    \\
    & = G^l_{k1}\tfrac{\partial}{\partial x^l}
    \wedge\tfrac{\partial}{\partial x^2} + \tfrac{\partial}{\partial
      x^1} \wedge G^m_{k2}\tfrac{\partial}{\partial x^m} =
    \left(G^1_{k1} + G^2_{k2}\right)\tfrac{\partial}{\partial x^1}
    \wedge\tfrac{\partial}{\partial x^2}
  \end{alignat*}
  we obtain
  \begin{displaymath}
    \n_W\xi = \left(G^1_{k1} + G^2_{k2}\right)W^k R(U,V)= \left(G^1_{k1}
      + G^2_{k2}\right)W^k\xi. 
  \end{displaymath}
  Since the geodesic coefficients are given by (\ref{eq:proj_flat_G_i}) 
  we have 
  \begin{equation}
  \label{n1xi}
  \n_W\xi = G^{m}_{km}W^{k}\xi = 
  3\frac{\partial \P}{\partial y^{k}}W^{k}\xi.
\end{equation}
Hence 
\begin{displaymath}
  \n_Z\left(\n_W\xi\right) = 3\n_Z\left(\frac{\partial \P}
    {\partial y^{k}}W^{k}\xi\right) = 3\left\{\n_Z\left(\frac{\partial \P}
      {\partial y^{k}}W^{k}\right)\xi + \left(\frac{\partial \P}
      {\partial y^{k}}W^{k}\right)\left(\frac{\partial \P}
      {\partial y^{l}}Z^{l}\right)\right\}\xi.
\end{displaymath}
Let $W$ be a vector field with constant coordinate functions. Then,
using (\ref{eq:proj_flat_G_i}) we get
\begin{displaymath}
  \nabla_Z \left(\frac{\partial \P}{\partial y^{k}}W^{k}\right)
  = \left(\frac {\partial^2 \P}{\partial x^{j}\partial y^{k}} 
    - G_j^{m}\frac{\partial^2 \P}{\partial y^{m}\partial y^{k}}
    \! \right)W^{k}Z^{j} 
  = \left(\frac{\partial^2 \P}{\partial x^{j}\partial y^{k}} -
    \P\frac{\partial^2 \P}{\partial y^{k}\partial
      y^{j}}\right)W^{k}Z^{j},
\end{displaymath}
and hence 
\begin{equation}
  \label{n2xi}
  \n_Z\left(\n_W\xi\right) = 
  3\left\{\frac {\partial^2 \P}{\partial x^{j}\partial y^{k}} - 
    \P\frac{\partial^2 \P}{\partial y^{k}\partial y^{j}} + \frac{\partial \P}
    {\partial y^{k}} \frac{\partial \P}{\partial y^{j}}\right\} 
  W^{k}Z^{j}\xi.
\end{equation}
Let $x_0\in M$ be the point with
  coordinates   $(0,0)$ in the local coordinate system of $(M, \F)$ 
  corresponding to the canonical coordinates of the projectively related 
  Euclidean plane. 
  According to Lemma 8.2.1 in \cite{ChSh}, p.155, we have
  \begin{displaymath}
    \frac {\partial^2 \P}{\partial x^1\partial y^2} - 
    \P\frac{\partial^2 \P}{\partial y^1\partial y^2} + \frac{\partial \P}
    {\partial y^1}\frac{\partial \P}{\partial y^2} = 2\frac{\partial \P}
    {\partial y^1}\frac{\partial \P}{\partial y^2} -
    \frac{l}{2}\frac{\partial^2 \F^2}{\partial y^1\partial y^2} =
    2\frac{\partial \P} {\partial y^1}\frac{\partial \P}{\partial y^2} -
    \lambda\;g_{12}.
  \end{displaymath}
  Hence  the vector fields $\xi\big|_{x_0}$, $\n_1\xi|_{x_0}$, $\n_2\xi|_{x_0}$
  and $\n_1\left(\n_2\xi\right)|_{x_0}$ are linearly independent if and
  only if the functions
  \begin{equation}
    \label{indep}
    \!1,\qquad \!\frac{\partial \P}{\partial y^1}\Big|_{x_0},\qquad 
    \! \frac{\partial \P}{\partial y^2}\Big|_{x_0},\qquad 
    \left(2\frac{\partial \P} {\partial y^1}
    \frac{\partial \P}{\partial y^2} -
    \lambda\;g_{12}\right)\Big|_{x_0}
  \end{equation}
  are linearly independent, where $g_{12}=g_y(\frac{\partial}{\partial
    x^1}, \frac{\partial}{\partial x^2})$ is the component of the metric
  tensor of $(M,\F)$.

\begin{lemma} 
  \label{expr}
  The functions\; $\frac{\partial \P(0,y)}{\partial y^1}$,
  $\frac{\partial \P(0,y)}{\partial y^2}$ and $\P(0,y)\frac{\partial
    ^2\P(0,y)}{\partial y^1\partial y^2}$ can be expressed in the polar
  coordinate system $(e^r,t)$ by 
  \[\frac{\partial\P(0,y)}{\partial y^1} \!=\! (\cos t + \dot r \sin t)e^{-r }\!,
    \quad     
    \frac{\partial\P(0,y)}{\partial y^2} \!=\! (\sin t - \dot r  \cos t)e^{-r}\!,\]
    \[\P(0,y)\frac{\partial ^2\P(0,y)}{\partial y^1\partial y^2}
    \!=\!  (\dot r^2 + 1-\ddot r )e^{-2r} \! \sin t \cos t ,\] where the
    dot refers to differentiation with respect to the variable $t$.
\end{lemma} 

\begin{proof} 
  We obtain from
  \begin{math}
    \frac{\partial e^{-r}}{\partial y^1} = - e^{-r} \dot r
    \frac{\partial t}{\partial y^1}
  \end{math}
  and from
  \begin{math}
    - \frac{y^2}{(y^1)^2} = \frac{\partial}{\partial
      y^1}(\frac{y^2}{y^1})\frac{d \tan t}{dt}\frac{\partial t}{\partial
      y^1} = \frac{1}{\cos^2 t}\frac{\partial t}{\partial y^1}
  \end{math}
  that
  \begin{math}
    \frac{\partial e^{-r}}{\partial y^1} = e^{-r}\dot r \cos^2 t
    \frac{y^2}{(y^1)^2} = e^{-r}\dot r\frac{y^2}{(y^1)^2 + (y^2)^2}.
  \end{math}
  Hence
  \begin{displaymath}
    \frac{\partial \P(0,y)}{\partial y^1} = 
    \frac{\partial \big(e^{-r}\sqrt{(y^1)^2+(y^2)^2}\big)}{\partial y^1} = 
    e^{-r}\Big(\dot r\frac{y^2}{\sqrt{(y^1)^2+(y^2)^2}} + 
    \frac{y^1}{\sqrt{(y^1)^2+(y^2)^2}}\Big).
  \end{displaymath}
  Similarly, we have
  \begin{math}
    \frac{\partial e^{-r}}{\partial y^2} = -
    e^{-r}\dot r \cos^2 t \frac{1}{y^1} = - e^{-r}\dot r\frac{y^1}{(y^1)^2
      + (y^2)^2}.
  \end{math}
  Hence
  \begin{displaymath}
    \frac{\partial \P(0,y)}{\partial y^2} = 
    \frac{\partial \big(e^{-r}\!\!\sqrt{(y^1)^2+(y^2)^2}\big)}{\partial y^2} 
    = e^{-r}\Big(-\dot r\frac{y^1}{\sqrt{(y^1)^2+(y^2)^2}} +
    \frac{y^2}{\sqrt{(y^1)^2+(y^2)^2}} \Big).
  \end{displaymath}
  Finally we have
  \begin{displaymath}
    \frac{\partial ^2\P(0,y)}{\partial y^1\partial y^2} = 
    \frac{\partial (\sin t - \dot r \cos t)e^{-r}}{\partial y^1} = 
    (\ddot r - \dot r^2 - 1 )e^{-r}\sin t \cos t \frac{1}{\sqrt{(y^1)^2+(y^2)^2}}. 
  \end{displaymath}
  Replacing $\varphi$ by the function $\P(0,y)$ in the expression (\ref{varphi}) 
  we get the assertion.
\end{proof}

\begin{lemma} 
  \label{non-linear} 
   Let $r:\mathbb R \to\mathbb R$ be a $2\pi$-periodic smooth function such that 
   the inequality $\ddot r(t) - \dot r^2(t) - 1\neq 0$ holds on a dense subset of 
   $\mathbb R$. Then the functions
  \begin{equation}\label{functions}
    1,\; (\cos t + \dot r \sin t)e^{-r},\; (\sin t - \dot r \cos
    t)e^{-r},\; (\cos t + \dot r \sin t)(\sin t - \dot r \cos t)e^{-2r}
  \end{equation}
  are linearly independent.
\end{lemma} 
\begin{proof} 
  The derivative of $(\cos t + \dot r \sin t)e^{-r}$ and of $(\sin t -
  \dot r \cos t)e^{-r}$ are $(\ddot r - \dot r^2 - 1)e^{-r}\sin t$ and
  $(\ddot r - \dot r^2 - 1)e^{-r}\cos t$, respectively, hence the
  functions (\ref{functions}) do not vanish identically.  Let us consider a
  linear combination
  \begin{displaymath}
    A + B (\cos t + \dot r \sin t)e^{-r } 
    + C (\sin t - \dot r \cos t)e^{-r } 
    +D(\cos t + \dot r \sin t)(\sin t - \dot r \cos t)e^{-2r} = 0
  \end{displaymath}
  with constant coefficients $A,B,C,D$. We differentiate and divide by 
  $e^{-t}(\ddot r - \dot r^2 - 1)$ and we have 
  \begin{displaymath}
    B \sin t - C \cos t - D (\cos 2t + \dot r \sin 2t)e^{-r}= 0.
  \end{displaymath}
  Putting $t = 0$ and $t = \pi$ we get $C = -De^{-r(0)} = De^{-r(\pi)}$. 
  Since $e^{-r(0)},\;e^{-r(\pi)} > 0$ we get $C = D = 0$   and hence 
  $A = B = C = D = 0$.
\end{proof}

Now, assume that condition (A) of Theorem
\ref{prop:proj_flat_const_curv} is fulfilled.  According to Proposition
\ref{sec:prop-1} if the functions (\ref{indep}) are linearly
independent, then the holonomy group $\mathsf{Hol}_{x_0}(M)$ is an
infinite dimensional subgroup of
${\mathsf{Diff}^{\infty}}({\I}_{x_0}M)$.  The function $\F({x_0},y)$
induces a scalar product on $T_{x_0}M$, consequently the component
$g_{12}$ of the metric tensor is constant on $T_{x_0}M$. Hence
$\mathsf{Hol}_{x_0}(M)$ is infinite dimensional if the functions
  \begin{equation}
    \label{functions_2}
    1, \qquad \frac{\partial \P}{\partial y^1}\Big|_{x_0},\qquad 
    \! \frac{\partial \P}{\partial y^2}\Big|_{x_0},\qquad 
    \frac{\partial \P} {\partial y^1}
    \frac{\partial \P}{\partial y^2}\Big|_{x_0} 
  \end{equation}
  are linearly independent. This follows from Lemma \ref{non-linear} 
  and hence the assertion of the theorem is true.\\[1ex]
  Assume that condition (B) is satisfied. We denote $\varphi(y) = \F(x_0,y)$. Using the expressions (\ref{indep}) 
  we obtain that the vector fields
  $\xi\big|_{x_0}$, $\n_1\xi|_{x_0}$, $\n_2\xi|_{x_0}$ and
  $\n_1\left(\n_2\xi\right)|_{x_0}$ are linearly independent if and only
  if the functions
  \begin{alignat*}{1}
    1,\qquad \!\frac{\partial\P}{\partial y^1}\Big|_{x_0} & =
    c\frac{\partial\varphi}{\partial y^1},\qquad \!
    \frac{\partial\P}{\partial y^2}\Big|_{x_0} =
    c\frac{\partial\varphi}{\partial y^2}
    \\
    \left(2\frac{\partial \P} {\partial y^1} \frac{\partial \P}{\partial
        y^2} - \lambda\;g_{12}\right)\Big|_{x_0} & = (2c^2 -
    \lambda)\frac{\partial\varphi} {\partial y^1}
    \frac{\partial\varphi}{\partial y^2} -
    \lambda\;\varphi\frac{\partial^2\varphi}{\partial y^1\partial y^2}
  \end{alignat*}
  are linearly independent. According to Lemma \ref{expr} this is
  equivalent to the linear independence of the functions
  \begin{displaymath}
    1,\qquad \!(\cos t + \dot r \sin t)\,e^{-r },
    \quad \! (\sin t - \dot
    r \cos t)\,e^{-r},
  \end{displaymath}
  \begin{displaymath}
    (2c^2 - \lambda)(\cos t + \dot r \sin t)(\sin t - \dot r \cos
    t)\,e^{-2r} - \lambda (\ddot r - \dot r^2 - 1)\,e^{-2r}\sin t \cos
    t.
  \end{displaymath}
  If $r = const$ then these functions are $1,\; \cos t\,e^{-r },\; \sin
  t\,e^{-r},\; 2c^2\cos t\sin t \,e^{-2r}$,\; hence the assertion follows 
  from Lemma \ref{non-linear}.  In the following we can assume that 
  $r(t) \neq const$. Let $t_0\in \mathbb R$
  such that $\dot r(t_0) = 0$ and $\kappa(t_0)\neq 0$. We rotate the
  coordinate system at the angle $- t_0$ with respect to the euclidean
  norm $\sqrt{(y^1)^2 + (y^2)^2}$, then we get in the new polar
  coordinate system that $\dot r(0) = 0$ and $\kappa(0)\neq 0$. Consider
  the linear combination
  \begin{equation}
    \label{eq:abcd}
    \begin{aligned}
      A & \!+\! B(\cos t + \dot r \sin t)e^{-r } + C(\sin t - \dot r
      \cos t)\,e^{-r} +
      \\
      & \!+\!D\big((2c^2 \!-\! \lambda)(\cos t \!+\! \dot r \sin t)(\sin
      t - \dot r \cos t)\,e^{-2r} -  \!\lambda (\ddot r \!-\!
      \dot r^2 \!-\! 1)\,e^{-2r}\sin t \cos t\big) \!=\! 0
    \end{aligned}
  \end{equation}
  with some constants $A,B,C,D$. Since the function $\varphi$ is
  absolutely homogeneous, the function $r(t)$ is $\pi$-periodic.
  Putting $t + \pi$ into $t$, the value of
  \[A + D(2c^2 - \lambda)(\cos t + \dot r \sin t)(\sin t - \dot r \cos
  t)\,e^{-2r} - \lambda (\ddot r - \dot r^2 - 1)\,e^{-2r}\sin t \cos t\]
  does not change, but the value of
  \[B(\cos t + \dot r \sin t)\,e^{-r } + C(\sin t - \dot r \cos
  t)\,e^{-r}\] changes sign. Since Lemma \ref{non-linear} implies that
  $(\cos t + \dot r \sin t)\,e^{-r }$ and $(\sin t - \dot r \cos
  t)\,e^{-r}$ are linearly independent, we have $B = C = 0$ and
  (\ref{eq:abcd}) becomes 
  \begin{equation}
    \label{eq:ad}
    A\,e^{2r} + D\Big((2c^2 - \lambda)\big[- \dot r \cos 2t + \frac12(1
    - \dot r^2)\sin 2t\big] - \frac{\lambda}{2} (\ddot r - \dot r^2 -
    1)\sin 2t\Big) = 0.
  \end{equation}
  Since  $\dot r(0) = 0$ at $t = 0$, we have $A = 0$. If $D\neq 0$ then
  (\ref{eq:ad}) gives
  \begin{displaymath}
    (2c^2 - \lambda)\big[- \dot r \cos 2t + \frac12(1 - \dot r^2)\sin 2t\big] 
    - \frac{\lambda}{2} (\ddot r - \dot r^2 - 1)\sin 2t = 0.
  \end{displaymath}
  By derivation and putting $t = 0$ we obtain
  \[(2c^2 - \lambda)\big[- \ddot r(0) + 1\big] - \lambda(\ddot
  r(0) - 1) = 2c^2(1 - \ddot r(0)) = 0.\] Using the relation (\ref{kappa})
  condition (B) gives $\kappa(0) = e^{r(0)}(1-\ddot r(0))\neq 0$, which is a
  contradiction.  Hence $D=0$ and the vector fields $\xi\big|_{x_0}$,
  $\n_1\xi|_{x_0}$, $\n_2\xi|_{x_0}$ and
  $\n_1\!\left(\n_2\xi\right)|_{x_0}$ are linearly independent. Using 
  Proposition \ref{sec:prop-1} we obtain the assertion.\\[1ex]
  Suppose now that the condition (C) holds.  Hence we have
  \begin{displaymath}
    \frac{\partial \F}{\partial y^1}(0,y) = 
    \frac{y^1}{|y|} \pm a^1,\qquad \frac{\partial \F}{\partial y^2}(0,y) =
    \frac{y^2}{|y|} \pm a^2,\qquad \frac{\partial^2 \F}{\partial y^1\partial
      y^2}(0,y) = - \frac{y^1y^2}{|y|^3},
  \end{displaymath}
  and
  \begin{equation}
    \label{Fexpr} 
    g_{12} = \left(\frac{y^1}{|y|} \pm 
      a^1\right)\left(\frac{y^2}{|y|} \pm a^2\right) - \left(1 \pm  
      \Big\langle a,\frac{y}{|y|} \Big\rangle\right)\frac{y^1y^2}{|y|^2}. 
  \end{equation}
  Similarly, we obtain from condition (C) that
  \begin{displaymath}
    \frac{\partial \P}{\partial y^1}(0,y) = \pm\frac{y^1}{|y|} - a^1,
    \quad
    \frac{\partial \P}{\partial y^2}(0,y) = \pm\frac{y^2}{|y|} - a^2.
  \end{displaymath}
  Using the expressions (\ref{indep}) we get that the vector fields 
  $\xi\big|_{x_0}$, $\n_1\xi|_{x_0}$, $\n_2\xi|_{x_0}$, $\n_1\left(\n_2\xi\right)|_{x_0}$ 
  are linearly independent if and only if the functions
  \begin{displaymath}
    1,\qquad \!\frac{\partial\P}{\partial y^1}\Big|_{(0,y)} = 
    \pm\frac{y^1}{|y|} - a^1,\qquad  \! 
    \frac{\partial\P}{\partial y^2}\Big|_{(0,y)} = \pm\frac{y^2}{|y|} - a^2
  \end{displaymath}
  and
  \begin{alignat*}{1}
    2\frac{\partial \P} {\partial y^1} \frac{\partial \P}{\partial y^2}
    \!-\!  \lambda\;g_{12}\Big|_{(0,y)} \!\!\!\!=
    % \\
    \mp \Big \langle \! a,\frac{y}{|y|} \Big\rangle\frac{y^1y^2}{|y|^2}
    \!+\! (1 \!-\! \lambda)\frac{y^1y^2}{|y|^2} \mp (2 \!+\! \lambda)
    \left(\!a_2 \frac{y^1}{|y|} \!+\! a_1 \frac{y^2}{|y|}\right) \!+\!
    (2 \!-\!  \lambda)a_1a_2
  \end{alignat*}
  are linearly independent.  Putting
  \begin{displaymath}
    \cos t = \frac{y^1}{|y|}, \quad \sin t = \frac{y^2}{|y|} 
  \end{displaymath}
  we obtain that this condition is true, since the trigonometric polynomials
  \[1,\quad \cos t,\quad \sin t, \quad (1 - \lambda)\cos t\sin t \mp(
  a_1\cos t + a_2 \sin t)\cos t\sin t \] 
  are linearly independent. Hence $\mathsf{Hol}_{x_0}(M)$ is 
  infinite dimensional.
\end{proof}

\section{Holonomy of projective Finsler manifolds of constant curvature}
\label{chap:4}

Now we will prove that the infinitesimal holonomy algebra of a totally geodesic 
submanifold of a Finsler manifold can be embedded into the infinitesimal holonomy 
algebra of the entire manifold. This result yields a lower estimate for the
 dimension of the holonomy group.

\subsubsection*{Totally geodesic and auto-parallel submanifolds} 

Let $(M,\S)$ be a spray manifold. A submanifold $\bar{M}$ is called
\emph{totally geodesic} if any geodesic of $(M,\S)$ which is tangent to
$\bar{M}$ at some point
is contained in $\bar{M}$. \\[1ex]
A totally geodesic submanifold $\bar{M}$ of the spray manifold $(M,\S)$
is called \emph{auto-parallel} if the homogeneous (nonlinear) parallel
translations $\tau_{c}:T_{c(0)}M\to T_{c(1)}M$ along curves in the
submanifold $\bar{M}$ leave invariant the tangent bundle $T\bar{M}$ and
for every $\xi\in\hat{\mathfrak X}^{\infty}(T\bar{M})$ the horizontal
covariant
derivative $\nabla_X\xi$ belongs to $\hat{\mathfrak X}^{\infty}(T\bar{M})$.\\[1ex]
Let $X, Y\in T_xM$ be tangent vectors at $x\in M$ and let $K$ denote the
curvature tensor of $(M,\S)$, (cf. equation (\ref{spraycurv})). The
mapping $y\to K(X,Y)(x,y):T_x M\mapsto T_xM$ is called \emph{curvature
  vector field at $x$} of the spray manifold $(M,\S)$.
\begin{lemma} \label{extend}
  Let $\bar{M}$ be a totally geodesic submanifold in a spray manifold $(M,\S)$. 
  The following assertions hold: \vspace{-5pt}
  \begin{enumerate} 
  \item[(a)] the spray $\S$ induces a spray $\bar{\S}$ on the
    submanifold $\bar{M}$, \vspace{-5pt}
  \item[(b)] $\bar{M}$ is an auto-parallel submanifold, \vspace{-5pt}
  \item[(c)] the curvature vector fields at any point of $\bar{M}$ can
    be extended to a curvature vector field of $M$. 
  \end{enumerate} 
\end{lemma} 
\begin{proof}  
  Assume that the manifolds $\bar{M}$ and $M$ are $k$, respectively
  $n=k+p$ dimensional.  Let $(x^1,\dots ,x^k,x^{k+1}, \dots,x^n)$ be an
  adapted coordinate system, i. e.  the submanifold $\bar{M}$ is locally
  given by the equations $x^{k+1}=\dots =x^n=0$.  We denote the indices
  running on the values $\{1,\dots ,k\}$ or $\{k + 1,\dots ,n\}$ by
  $\alpha, \beta, \gamma$ or $\sigma, \tau$, respectively. The
  differential equation (\ref{geod}) of geodesics yields that the
  geodesic coefficients $\Gamma^{\sigma}(x,y)$ satisfy
  \[\Gamma^{\sigma}(x^1,\dots ,x^k,0,\dots ,0;y^1,\dots ,y^k,0,\dots ,0) = 0\] 
  identically, hence their derivatives with respect to $y^1,\;\dots
  ,\;y^k$ are also vanishing. It follows that $\Gamma^{\sigma}_{\alpha}
  = 0$ and $\Gamma^{\sigma}_{\alpha\;\beta} = 0$ at any $(x^1,\dots
  ,x^k,0, \dots ,0;y^1,\dots ,y^k,0,\dots ,0)$. Hence the induced spray
  $\bar{\S}$ on $\bar{M}$ is defined by the geodesic coefficients
  \[\bar{\Gamma}^{\beta}(x^1,\dots ,x^k;y^1,\dots ,y^k) = 
  \Gamma^{\beta}(x^1,\dots ,x^k,0,\dots ,0;y^1,\dots ,y^k,0,\dots ,0).\]  
  The homogeneous (nonlinear) parallel translation $\tau_{c}:T_{c(0)}M\to T_{c(1)}M$ along 
  curves in the submanifold $\bar{M}$ and the horizontal covariant derivative on $\bar{M}$
  with respect to the spray $\S$ coincide with the translation and the horizontal covariant 
  derivative on $\bar{M}$ with respect to the spray $\bar{\S}$. Hence the first two 
  assertions are true. \\ 
  If $y, X, Y\in T_x\bar{M}$ are tangent vectors at $x\in \bar{M}$ 
  then  $K(X,Y)(x,y)$ can be expressed by 
  \[\left(\frac{\partial \Gamma^i_{\alpha}(x,y)}{\partial x^{\beta}} -
  \frac{\partial \Gamma^i_{\beta}(x,y)}{\partial x^{\alpha}} + \Gamma_{\alpha}^m(x,y)
  \Gamma^i_{{\beta} m}(x,y) - \Gamma_{\beta}^m(x,y)
  \Gamma^i_{{\alpha} m}(x,y)\right)X^{\alpha}Y^{\beta}\frac{\partial}{\partial x^i}. \]  
  Since $\Gamma^{\sigma}_{\alpha} = 0$ and   $\Gamma^{\sigma}_{\alpha\;\beta} = 0$ at any $(x^1,\dots ,x^k,0,
  \dots ,0;y^1,\dots ,y^k,0,\dots ,0)$ we have 
   \[\frac{\partial \Gamma^{\sigma}_{\alpha}}{\partial x^{\beta}} -
  \frac{\partial \Gamma^{\sigma}_{\beta}}{\partial x^{\alpha}} + 
  \Gamma_{\alpha}^\tau
  \Gamma^{\sigma}_{{\beta}\tau} - \Gamma_{\beta}^\tau
  \Gamma^{\sigma}_{{\alpha} \tau}+ 
  \Gamma_{\alpha}^\gamma
  \Gamma^{\sigma}_{{\beta}\gamma} - \Gamma_{\beta}^\gamma
  \Gamma^{\sigma}_{{\alpha} \gamma} = 0\]
  at $(x^1,\dots ,x^k,0,\dots ,0;y^1,\dots ,y^k,0,\dots ,0)$.
  Hence the curvature tensors $\bar{K}$ corresponding to the spray $\bar{\S}$ and $K$ 
  corresponding to the spray $\S$ satisfy $\bar{K}(X,Y)(x,y) = K(X,Y)(x,y)$ if 
  $x\in\bar M$ and  $y, X, Y\in T_x\bar{M}$. It follows that for any given 
  $X, Y\in T_x\bar{M}$ the curvature vector field $\bar{\xi}(y) = \bar{K}(X,Y)(x,y)$ 
  at $x\in\bar M$ defined on $T_x\bar{M}$ can be extended to the curvature vector 
  field $\xi(y) = K(X,Y)(x,y)$ at $x\in\bar M$ defined on $T_xM$.
\end{proof}

\begin{theorem}
  \label{simultan}
  Let $\bar{M}$ be a totally geodesic 2-dimensional submanifold of a Finsler 
  manifold $(M, \F)$ such that the infinitesimal holonomy algebra 
  $\mathfrak{hol}^{*}_x(\bar{M})$ of $\bar{M}$ is infinite dimensional. Then the 
  holonomy group $\mathsf{Hol}_x(M)$ is infinite dimensional.
\end{theorem} 

\begin{proof} 
  According to Lemma \ref{extend} any curvature vector field of
  $\bar{M}$ at $x\in \bar{M}\subset M$ defined on $\I_x \bar{M}$ can be
  extended to a curvature vector field on the indicatrix $\I_x M$. Hence
  the curvature algebra $\mathfrak{R}_x(\bar{M})$ of the submanifold
  $\bar{M}$ can be embedded into the curvature algebra
  $\mathfrak{R}_x(M)$ of the manifold $(M, \F)$. Assume that $\bar{\xi
  }$ is a vector field belonging to the infinitesimal holonomy algebra
  $\mathfrak{hol}^{*}_x(\bar{M})$ which can be extended to the vector
  field $\xi$ belonging to the infinitesimal holonomy algebra
  $\mathfrak{hol}^{*}_x(M)$. Any a vector field $\bar{X}\!\in\!
  {\mathfrak X}^{\infty}(\bar{M})$ can be extended to a vector field
  $X\!\in\!  {\mathfrak X}^{\infty}(M)$, hence the Berwald horizontal
  covariant derivative along $\bar{X}\in {\mathfrak
    X}^{\infty}(\bar{M})$ of $\bar{\xi }$ can be extended to the Berwald
  horizontal covariant derivative along $X\in {\mathfrak X}^{\infty}(M)$
  of the vector field $\xi$. It follows that the infinitesimal holonomy
  algebra $\mathfrak{hol}^{*}_x(\bar{M})$ of the submanifold $\bar{M}$
  can be embedded into the infinitesimal holonomy algebra
  $\mathfrak{hol}^{*}_x(M)$ of the Finsler manifold
  $(M,\F)$. Consequently, $\mathfrak{hol}^{*}_x(M)$ is infinite
  dimensional and hence the holonomy group $\mathsf{Hol}_x(M)$ is an
  infinite dimensional subgroup of ${\mathsf{Diff}^{\infty}}({\I}_xM)$.
\end{proof}
This result can be applied to locally projectively flat Finsler manifolds, as they 
have for each tangent $2$-plane a totally geodesic submanifold which is tangent to 
this $2$-plane.
\begin{corollary} 
  \label{idhg} 
  If a locally projectively flat Finsler manifold has a   2-dimensional totally 
  geodesic submanifold satisfying one of the conditions of Theorem 
  \ref{prop:proj_flat_const_curv}, then its holonomy group is infinite dimensional.
\end{corollary}
According to equations (\ref{projective1}) and (\ref{projective2}) the projectively 
flat Randers manifolds of non-zero constant curvature satisfy condition (C) of 
Theorem \ref{prop:proj_flat_const_curv}. We can apply Corollary \ref{idhg} to 
these manifolds and we get the following 

\begin{theorem}
  The holonomy group of any projectively flat Randers manifolds of non-zero 
  constant flag curvature is infinite dimensional.
\end{theorem}
R. Bryant in [Br1], [Br2] introduced and studied complete Finsler
metrics of positive curvature on $S^2$. He proved that there exists
exactly a $2$-parameter family of Finsler metrics on $S^2$ with
curvature = $1$ with great circles as geodesics. Z. Shen generalized a
$1$-parameter family of complete Bryant metrics to $S^n$ satisfying
\begin{equation}\label{BSh}
\F(0,y) = |y| \,\cos \alpha, \quad \P(0,y) = |y|\,\sin \alpha 
\end{equation}
with $|\alpha|<\frac{\pi}{2}$ in a coordinate neighbourhood centered at
$0\in \mathbb R^n$, (cf. Example 7.1. in \cite{Shen4} and Example 8.2.9
in
\cite{ChSh}).\\
We investigate the holonomy groups of two families of metrics,
containing the $1$-parameter family of complete Bryant-Shen metrics
(\ref{BSh}). The first family in the following theorem is defined by
condition (A), which is motivated by Theorem
8.2.3 in \cite{ChSh}. There is given the following construction:\\
If $\psi = \psi(y)$ is an arbitrary Minkowski norm on $\mathbb R^n$ and
$\varphi = \varphi(y)$ is an arbitrary positively 1-homogeneous function
on $\mathbb R^n$, then there exists a projectively flat Finsler metric
$\F$ of constant flag curvature $-1$, defined on a neighbourhood of the
origin, such that $\F$ and its projective factor $\P$ satisfy
$\F(0,y) = \psi(y)$ and $\P(0,y) = \varphi(y)$. \\
Condition (B) in the next theorem is confirmed by Example 7 in
\cite{Shen4}, p.~1726, where it is proved that for an arbitrary given
Minkowski norm $\varphi$ and $|\vartheta|<\frac{\pi}{2}$ there exists a
projectively flat Finsler function $\F$ of constant curvature $= 1$
defined on a neighbourhood of $0 \in \mathbb R^n$, such that
\begin{displaymath}
  \F(0,y)= \varphi(y)\,\cos \vartheta   \quad \mathrm{and}
  \quad   \P(0,y) = \varphi(y)\,\sin\vartheta.
\end{displaymath}
Conditions (A) and (B) in Theorem \ref{prop:proj_flat_const_curv}
together with Corollary \ref{idhg} yield the following
\begin{theorem}
 Let $(M, \F)$ be a projectively flat Finsler manifold of non-zero
  constant curvature. Assume that there exists a point $x_0\in M$ and a
  $2$-dimensional totally geodesic submanifold $\bar{M}$ through $x_0$ 
  such that one of the following conditions holds \vspace{-5pt}
  \begin{enumerate}
  \item[\emph{(A)}] $\F$ induces a scalar product on $T_{x_0}\bar{M}$,
    and the projective factor $\P$ on $T_{x_0}\bar{M}$ is a strongly
    convex positively 1-homogeneous function, \vspace{-5pt}
  \item[\emph{(B)}] $\F(x_0,y)$ on $T_{x_0}\bar{M}$ is a strongly convex
    absolutely 1-homogeneous function on $T_{x_0}M$, and the projective
    factor $\P({x_0},y)$ on $T_{x_0}\bar{M}$ satisfies
    $\P({x_0},y)=c\!\cdot\!\F(x_0,y)$ with $0\neq c\in\mathbb
    R$. \vspace{-5pt}
  \end{enumerate} 
  Then the holonomy group $\mathsf{Hol}_{x_0}(M)$ is infinite  dimensional.
 \end{theorem}

\end{document}